\documentclass[12pt]{amsart}

\usepackage{amssymb,amsthm,amsmath,amscd}

\usepackage{mathtools}
\usepackage{soul}

\RequirePackage[dvipsnames,usenames]{color}

\input{kmacros3.sty}

\usepackage{tikz}
\usepackage{tikz-cd}
\usetikzlibrary{cd}
\usepackage{graphicx}
\usepackage[all,cmtip]{xy}

\usepackage{mabliautoref}
\usepackage{bm}
\usepackage{pifont}
\usepackage{upgreek}

\usepackage{eucal}
\usepackage{ulem}
\usepackage{stmaryrd}

\usepackage{fullpage}

\usepackage{setspace}
\usepackage{enumerate}
\usepackage{calc}

\usepackage{bbding}

\usepackage{verbatim}
\usepackage{alltt}

\newcommand{\NN}{\mathbb{N}}

\newcommand{\ux}{\underline{x}}
\newcommand{\un}{\underline{n}}

\newcommand{\ds}{\displaystyle}

\begin{document}

 \title{Primary Decompositions of Regular Sequences}
 \author{Thomas Polstra}
 \address{Department of Mathematics, University of Alabama, Tuscaloosa, AL 35487 USA}
\email{tmpolstra@ua.edu}
 \thanks{Polstra was supported in part by NSF Grant DMS \#2101890}
\maketitle

\begin{abstract}
Let $R$ be a Noetherian ring and $x_1,\ldots,x_t$ a permutable regular sequence of elements in $R$. Then there exists a finite set of primes $\Lambda$ and natural number $C$ so that for all $n_1,\ldots,n_t$ there exists a primary decomposition $(x_1^{n_1},\ldots,x_t^{n_t})=Q_1\cap \cdots \cap Q_\ell$ so that $\sqrt{Q_i}\in \Lambda$ and $\sqrt{Q_i}^{C(n_1+\cdots + n_t)}\subseteq Q_i$ for all $1\leq i\leq \ell$.
\end{abstract}

\section{Introduction}

Primary decompositions of ideals in commutative algebra correspond to decompositions of closed subschemes into irreducible subspaces in algebraic geometry. Let $R$ be a Noetherian ring and $I\subseteq R$ an ideal. By the main result of \cite{Swanson} there exists an integer $C$ so that for every $n\in \mathbb{N}$ there exists a primary decomposition 
\[
I^n=Q_1\cap \cdots \cap Q_\ell
\]
so that $\sqrt{Q_i}^{Cn}\subseteq Q_i$ for all $1\leq i\leq \ell$. Swanson's theorem has applications to multiplicity theory, \cite{Cutkosky1,Cutkosky2,Cutkosky3,Das,Cid-RuizMontano}, the uniform symbolic power topology problem, \cite{SwansonLinear,EinLazarsfeldSmith,HochsterHuneke,HKV,HK,GHM}, and localizations problems in tight closure theory, \cite{HunekeSaturation,Vraciu,SmithSwanson,Dinh}.

Swanson's proof proceeds by first reducing to the scenario that $I$ is principally generated by a nonzerodivisor. Indeed, the extended Rees algebra $S:=R[It,t^{-1}]$ enjoys the property that $t^{-1}$ is a nonzerodivisor and $I^n=(t^{-1}S)^n\cap R$. If $(t^{-1}S)^n=Q_1\cap \cdots \cap Q_\ell$ is a suitable primary decomposition of $(t^{-1}S)^n$ then $I^n=(Q_1\cap R)\cap \cdots \cap (Q_\ell\cap R)$ is a primary decomposition of $I^n$ with the desired properties. Our main result extends Swanson's Theorem to ideals generated by a permutable regular sequence.

\begin{theorem}
    \label{main theorem}
    Let $R$ be a Noetherian ring and $x_1,\ldots,x_t$ a permutable regular sequence. There exists a finite set of primes $\Lambda$ and a constant $C$ so that for every $n_1,\ldots, n_t\in\mathbb{N}$ there exists a primary decomposition
    \[
    (x_1^{n_1},\ldots,x_t^{n_t})=Q_1\cap \cdots \cap Q_\ell
    \]
    so that $\sqrt{Q_i}\in\Lambda$ and $\sqrt{Q_i}^{C(n_1+\cdots +n_t)}\subseteq Q_i$ for all $1\leq i\leq \ell$.
\end{theorem}

The methodology of \cite{Swanson} is akin to the techniques of Huneke's Uniform Artin-Rees Theorem, \cite{HunekeUniform}. Other's have re-proven Swanson's theorem without relying on such technicalities, see \cite{Sharp1,Sharp2,Yao1,Yao2} for more general decomposition statements involving products of powers of ideals $I_1^{n_1}\cdots I_t^{n_t}$ and their integral closures. Similar to their methods, the present article fundamentally depends only upon the standard Artin-Rees Lemma \cite[Proposition~10.9]{AtiyahMacdonald}, and the theory of injective hulls, \cite[Section~3.2]{BrunsHerzog}.

\section{Primary Decompositions of Regular Sequences}

Let $I\subseteq R$ be an ideal and $x\in R$ an element. By the Artin-Rees Lemma there exists a constant $C$ so that $(x)\cap P^{h+C}=((x)\cap P^C)P^h\subseteq xP^h$ for all $h$, see \cite[Proposition~10.9]{AtiyahMacdonald}.

\begin{lemma}
\label{lemma: Sharp's lemma}
Let $R$ be a Noetherian ring and $x\in R$ a non-unit. Let $P\in \Spec(R)$ and $E=E_R(R/P)$. Let $C$ be chosen such that $(x)\cap P^{h+C}\subseteq xP^h$ for all $h\in \mathbb{N}$. If $\varphi: R\to E$ is an $R$-linear map with the property that $P^h\varphi=0$ then there exists an $R$-linear map $\psi:  R\to E$ such that 
\begin{itemize}
    \item $\varphi=x\psi$;
    \item $P^{h+C}\psi=0$.
\end{itemize}
\end{lemma}

\begin{proof}
We are assuming that $C$ is chosen such that $(x)\cap P^{h+C}\subseteq xP^h$ for all integers $h$. In particular, there are natural surjections 
\[
\frac{R}{(x)\cap P^{h+C}}\to \frac{R}{xP^h}.
\]
Therefore there are inclusions
\[
\Hom_R\left(\frac{R}{xP^h}, E\right)\to \Hom_R\left(\frac{R}{(x)\cap P^{h+C}}, E\right).
\]
Equivalently,
\[
(0:_E xP^{h})\subseteq (0:_E ((x)\cap P^{h+C})).
\]
Even further, there are natural inclusions
\[
\frac{R}{(x)\cap P^{h+C}}\subseteq \frac{R}{(x)}\oplus \frac{R}{P^{h+C}}.
\]
Therefore there are natural surjections
\[
(0:_E (x)) + (0:_E P^{h+C})\twoheadrightarrow (0:_E ((x)\cap P^{h+C})).
\]
In conclusion, if 
\[
\lambda \in \Hom_R(R/xP^h,E)\cong (0:_ExP^h)\subseteq (0:_E ((x)\cap P^{h+C}))
\]
then there exists 
\[
\lambda'\in\Hom_R(R/(x),E) \cong (0:_E (x))
\]
and 
\[
\psi\in \Hom_R(R/P^{h+C},E)\cong (0:_E P^{h+C})
\]
such that $\lambda=\lambda'+\psi$.

The module $E$ is injective and therefore there exists $\lambda$ such that $\varphi=x\lambda$, i.e. the following diagram commutes:
\[
\begin{xymatrix}
{
\displaystyle 
R\ar[r]^{\cdot x}\ar[d]^{\varphi} & R\ar[ld]^{\lambda} \\ 
E
}
\end{xymatrix}
\]
Since $P^h\varphi=0$ we have that $xP^h\lambda =0$. We can therefore write $\lambda= \lambda'+\psi$ so that $x\lambda'=0$ and $P^{h+C}\psi=0$. Therefore $\varphi=x\lambda=x\psi$ and $\psi:R\to E$ enjoys the desired properties.
\end{proof}

Adopt the following notation: Let $\ux=x_1,\ldots, x_t$ be a sequence of elements of a Noetherian ring $R$ and $\un=(n_1,\ldots,n_t)\in \mathbb{N}^{\oplus t}$.

\begin{itemize}
    \item $\ux^{\un}=x_1^{n_1},\ldots, x_t^{n_t}$;
    \item $e_i\in \mathbb{N}^{\oplus t}$ is the element with a $1$ in the $i$th coordinate and $0$'s elsewhere;
    \item $\underline{1}=(1,\ldots, 1)\in \mathbb{N}^{\oplus t}$;
    \item If $\un'\in \mathbb{N}^{\oplus t}$ then $\un\cdot \un'$ denotes the dot product of $\un$ and $\un'$. In particular, the element $\un-(\un\cdot e_i-1)e_i$ is the element of $\mathbb{N}^{\oplus t}$ obtained by replacing the $i$th coordinate of $\un$ with the number $1$.
\end{itemize}
Observe that if $x_1,\ldots,x_t$ is a permutable regular sequence and $\un\in\NN^{\oplus t}$ then $(\underline{x}^{\un+e_i}):x_i=(\underline{x}^{\un})$.

\begin{theorem}
\label{Thm: embedding regular sequences into injective modules}
Let $R$ be a Noetherian ring and $\underline{x}=x_1,\ldots,x_t$ a permutable regular sequence. Fix a finite list of prime ideals $\Lambda$, allowing for the possibility of repeated primes in $\Lambda$, and an embedding 
\[
\frac{R}{(\underline{x})}\xhookrightarrow{\varphi_{\underline{1}}}\bigoplus_{P\in \Lambda}E_R(R/P).
\]
Let $C$ be chosen large enough so that $P^{C|\underline{1}|}\varphi_{\underline{1}}=0$ and $(x_i)\cap P^{h+C}\subseteq x_iP^h$ for all $P\in\Lambda$, $h\in\mathbb{N}$, and $1\leq i\leq t$. Then for all $\un\in\mathbb{N}^{\oplus t}$ there exists an embedding
\[
\frac{R}{(\underline{x}^{\underline{n}})}\xhookrightarrow{\varphi_{\underline{n}}}E_{\underline{n}}
\]
such that $E_{\underline{n}}\cong \left(\bigoplus_{P\in \Lambda}E_R(R/P)\right)^{\ell_{\un}}$ for some integer $\ell_{\un}$ and $P^{C|\un|}\varphi_{\un}=0$ for all $P\in\Lambda$.
\end{theorem}

\begin{proof}
By induction, we may suppose that we have constructed the injective module $E_{\un}$ for all $\un\leq \un'$ and maps $\varphi_{\underline{n}}:R/(\ux^{\un})\hookrightarrow E_{\un}$ such that $P^{C|\un|}\varphi=0$ for the purposes of constructing $E_{\underline{n}'+e_i}$ and map $\varphi_{\underline{n}'+e_i}: R/(\underline{x}^{\un'+e_i})\hookrightarrow E_{\un'+e_i}$ such that $P^{C|\un'+e_i|}\varphi_{\un'+e_i}=0$. Even further, we suppose that $E_{\un}$ consists of direct sums of $\bigoplus_{P\in \Lambda}E_R(R/P)$ for all $\un\leq \un'$.

Because $\underline{x}$ is a permutable regular sequence there exists short exact sequences
\[
0\to \frac{R}{(\underline{x}^{\un'})}\xrightarrow{\cdot x_i} \frac{R}{(\underline{x}^{\un +e_i})} \xrightarrow{\pi}\frac{R}{(\underline{x}^{\un'-(\un'\cdot e_i-1)e_i})}\to 0.
\]
Lemma~\ref{lemma: Sharp's lemma} applied to each of the irreducible direct summands of $E_{\un'}$ produces a map $\psi_{\underline{n}'}:R/(\ux^{\un'+e_i})\to E_{\un'}$ such that $P^{C|\underline{n}'+e_i|}\psi_{\underline{n}'}=0$ and the following diagram commutes:
\[
\begin{xymatrix}
{
0 \ar[r]& \ds \frac{R}{(\ux^{\un'})}\ar[rr]^{\cdot x_i}\ar@{^{(}->}[d]^{\varphi_{\un'}} && \ds \frac{R}{(\ux^{\un'+e_i})}\ar[lld]^{\psi_{\un'}}\ar[d]^{(\psi_{\un'},\varphi_{\un'-(\un'\cdot e_i-1)e_i}\circ\pi)} \ar[rr]^{\pi}&& \ds \frac{R}{(\ux^{\un'-(\un'\cdot e_i-1)e_i})}\ar[d]^{\varphi_{\un'-(\un'\cdot e_i-1)e_i}} \ar[r]& 0 \\ 
0 \ar[r]& E_{\un'}\ar[rr] && E_{\un'}\oplus E_{\un'-(\un'\cdot e_i-1)e_i}\ar[rr] && E_{\un'-(\un'\cdot e_i-1)e_i} \ar[r]& 0 
}
\end{xymatrix}
\]
It is straight-forward to verify that $\varphi_{\un'+e_i}:=(\psi_{\un'},\varphi_{\un'-(\un'\cdot e_i-1)e_i}\circ\pi)$ is an injective map and $P^{C|\un+e_i|}\varphi_{\un'+e_i}=0$.
\end{proof}

\begin{corollary}[Swanson's Theorem for regular sequences]
\label{Corollary: Swanson's Theorem for regular sequences}
Let $R$ be a Noetherian ring and $\underline{x}=x_1,\ldots,x_t$ a permutable regular sequence. There exists a finite set of primes $\Lambda$ and a constant $C$ such that for all $\underline{n}\in\mathbb{N}^{\oplus t}$ there exists a primary decomposition
\[
(\ux^{\un})=Q_1\cap \cdots \cap Q_\ell
\]
such that $\sqrt{Q_i}\in \Lambda$ and $\sqrt{Q_i}^{C|\un|}\subseteq Q_i$ for all $1\leq i \leq \ell$.
\end{corollary}

\begin{proof}
Fix $\un\in\mathbb{N}^{\oplus t}.$ By Theorem~\ref{Thm: embedding regular sequences into injective modules} there exists a constant $C$, not depending on $\un\in\mathbb{N}^{\oplus t}$, and a finite set of primes $\Lambda_{\un}$, allowing for the possibility of repeated primes in $\Lambda_{\un}$, and an embedding
\[
\varphi_{\un}: \frac{R}{(\ux^{\un})}\hookrightarrow \bigoplus_{P\in\Lambda_{\un}}E(R/P)
\]
such that $P^{C|\un|}\varphi_{\un}=0$ for all $P\in\Lambda_{\un}$. Let $\pi:R\to R/(\ux^{\un})$ and $\pi_P:\bigoplus_{P\in\Lambda_{\un}}E(R/P)\to E(R/P)$ be the natural surjections. Then 
\[
(\ux^{\un})=\bigcap_{P\in \Lambda}\Ker(\pi_P\circ \varphi_{\un}\circ \pi)
\]
is a primary decomposition of $(\ux^{\un})$ as there are embeddings
\[
\frac{R}{\Ker(\pi_P\circ \varphi_{\un}\circ \pi)}\hookrightarrow E(R/P).
\]
Furthermore, $P^{C|\un|}\subseteq\Ker(\pi_P\circ \varphi_{\un}\circ \pi)$ since $P^{C|\un|}\varphi=0$.
\end{proof}

\bibliographystyle{skalpha}
\bibliography{main}

\end{document}